\documentclass[12pt,leqno,fleqn]{amsart}  
\usepackage{amsmath,amstext,amsthm,amssymb,amsxtra}

\usepackage{mathtools} 

\usepackage{float}
\usepackage{amsrefs}
\usepackage{hyperref} 
\hypersetup{
    colorlinks=true,       
    linkcolor=blue,          
    citecolor=magenta,        
    filecolor=magenta,      
    urlcolor=cyan,          
}

\usepackage{pgf,pgfarrows} 
\usepackage{tikz}
\usetikzlibrary[decorations.pathreplacing]

\allowdisplaybreaks
\swapnumbers

\allowdisplaybreaks
\swapnumbers

\theoremstyle{plain} 
\newtheorem{lemma}[equation]{Lemma} 
\newtheorem{proposition}[equation]{Proposition} 
\newtheorem{theorem}[equation]{Theorem} 
 
\newtheorem{conjecture}[equation]{Conjecture}

\theoremstyle{definition}

\theoremstyle{remark}
\newtheorem{remark}[equation]{Remark}

\numberwithin{equation}{section}

\def\norm#1.#2.{\lVert#1\rVert_{#2}}
\def\Norm#1.#2.{\bigl\lVert#1\bigr\rVert_{#2}}
\def\NOrm#1.#2.{\Bigl\lVert#1\Bigr\rVert_{#2}}
\def\NORm#1.#2.{\biggl\lVert#1\biggr\rVert_{#2}}
\def\NORM#1.#2.{\Biggl\lVert#1\Biggr\rVert_{#2}}

\def\ip#1,#2,{\langle #1,#2\rangle}
\def\Ip#1,#2,{\langle#1,#2\rangle}
\def\IP#1,#2,{\langle#1,#2\rangle}

\def\mid{\,:\,}

\def\XXint#1#2#3{{\setbox0=\hbox{$#1{#2#3}{\int}$}
     \vcenter{\hbox{$#2#3$}}\kern-.5\wd0}}

\newcommand{\rt}{\rightarrow}

\newcommand{\nint}{ \displaystyle\int}

\newcommand{\norme}[1]{ \left\| #1 \right\|}

\newcommand{\ral}{\mathbb{R}}

\newcommand{\nat}{\mathbb{N}}

\begin{document}

\title[Maximal functions and Singular Integrals on Weighted spaces]{On joint estimates for maximal functions and singular integrals in weighted spaces}

\author[M.C. Reguera]{Maria Carmen Reguera}
\thanks{Research supported in part by grant NSF-DMS 0968499}
\address{School of Mathematics, Georgia Institute of Technology, Atlanta GA 30332, USA, and Centre for Mathematical Sciences, University of Lund, Lund, Sweden}
\email {mreguera@math.gatech.edu \\ mreguera@maths.lth.se}

\author[J. Scurry]{James Scurry}
\thanks{Research supported in part by the National Science Foundation under Grant No. 1001098}
\address{ School of Mathematics, Georgia Institute of Technology, Atlanta GA 30332, USA}
\email {jscurry3@math.gatech.edu}


\begin{abstract}
We consider a conjecture attributed to Muckenhoupt and Wheeden which suggests a positive relationship between the continuity of the Hardy-Littlewood maximal operator and the Hilbert transform in the weighted setting. Although continuity of the two operators is equivalent for $A_p$ weights with $1 < p < \infty$, through examples we illustrate this is not the case in more general contexts. In particular, we study weights for which the maximal operator is bounded on the corresponding $L^p$ spaces while the Hilbert transform is not. We focus on weights which take the value zero on sets of non-zero measure and exploit this lack of strict positivity in our constructions. These type of weights and techniques have been explored previously in Reguera \cite{1008.3943} and Reguera-Thiele \cite{1011.1767}.
\end{abstract}

\maketitle

\section{Introduction}
Characterizing boundedness of singular integral operators in the two weighted setting has proven to be a very delicate problem even for the particular instance of the Hilbert transform. Of the different aspects of this active subject, we are interested in the \emph{joint} boundedness of the Hardy-Littlewood maximal operator and the Hilbert transform for a certain pair of weights. Nazarov, Treil and Volberg characterize the two weight boundedness of the family formed by two Maximal operators and the Hilbert transform, \cite{1003.1596}. Our goal is to show that once we consider the two families independently, there is no correlation between boundedness of one or the other. For other aspects of the subject we refer the reader to the work of Lacey, Sawyer and Uriarte-Tuero \cites{0911.3437, 0807.0246, 0911.3920, 1001.4043, 1108.2319} and Cruz-Uribe, Martell and P\'erez \cite{MR2351373}.

For the Muckenhoupt one weight setting, the continuity of the Hilbert transform $H$ and Hardy-Littlewood maximal function $M$ are equivalent. Namely, for $w$ a weight in the Muckenhoupt $A_p$ class, with $1 < p < \infty$, we have $H: L^p(w) \rt L^p(w)$ iff  $M: L^p(w) \rt L^p(w)$. It is important to mention that weights in the $A_p$ class satisfy $w>0$ a.e. and this property is crucial, as we will see later. In \cite{MR2351373}, the authors state a conjecture attributed to Muckenhoupt and Wheeden that suggests a subtle but positive relationship between $M$ and $H$ in the two weight setting. We state the conjecture explicitly,

\begin{conjecture}[$L^{p}$ Muckenhoupt-Wheeden]\label{c.lpMW}
Let $M$ be the Hardy-Littlewood maximal operator, $T$ be a Calder\'on-Zygmund operator and let $w$ and $v$ be weights on $\mathbb R^{d}$. Then 
\begin{eqnarray}\label{e.2wM}
M:& L^{p}(v)\mapsto L^{p}(w)\\ \label{e.2wMdual}
M:& L^{p'}(w^{1-p'}) \mapsto L^{p'}(v^{1-p'})
\end{eqnarray}
if and only if
\begin{equation} \label{e.2wT}
T:  L^{p}(v)\mapsto L^{p}(w).
\end{equation}
\end{conjecture}

\noindent For precise definitions of Maximal and Hilbert we refer the reader to \eqref{e.maxf} and \eqref{e.hilbertt} respectively.


Even in the one weight setting, if we allow ourself to go outside the Muckenhoupt $A_p$ class (one should be careful with the way of defining $L^{p}(w)$ when the weight is compactly supported), we will find examples of weights $w$ for which the Hardy-Littlewood maximal $M$ is bounded but the Hilbert transform $H$ is not . The main results presented in the paper are listed below.

\begin{theorem}\label{t.maint2}
Let $1<p<\infty$. There exists a nontrivial weight $u$ for which the Hardy-Littlewood maximal operator $M$ is bounded from $ L^{p}(u)$ to $L^{p}(u)$ but the Hilbert transform $H$ is unbounded from $L^{p}(u)$ to $L^{p}(u)$.
\end{theorem}

\begin{theorem}\label{t.main}
Let $1<p<\infty$ and let $p'$ be the dual exponent, $\frac{1}{p}+\frac{1}{p'}=1$. There exist nontrivial weights $w$ and $v= \left (\frac{Mw}{w}\right)^{p} w$ for which the Hardy-Littlewood maximal operator satisfies

\begin{eqnarray}\label{e.2wMmain}
M:& L^p(\left (\frac{Mw}{w}\right)^{p} w) \mapsto L^{p}(w)\\ \label{e.2wMdualmain}
M:& L^{p'}(w^{1-p'}) \mapsto L^{p'}(\frac{w}{(Mw)^{p'}})
\end{eqnarray}

but the Hilbert transform $H$ is unbounded from $L^p(\left (\frac{M w}{w}\right)^{p} w)$ to $L^{p}(w)$.

\end{theorem}

\begin{theorem} \label{m.c.lsut}
There exist measures $\gamma$ and $\lambda$ such that 
\begin{eqnarray}
&M(\cdot \gamma): & L^2(\gamma) \not\rt L^2(\lambda) \label{cantor.max} \\
&H( \cdot \gamma): & L^2(\gamma) \rt L^2(\lambda) . \label{lsut.hilbert}
\end{eqnarray}
\end{theorem}

For the proof of Theorems \ref{t.main} and \ref{t.maint2}, we will be using the ideas of Reguera \cite{1008.3943} and Reguera-Thiele \cite{1011.1767}. In the papers the authors present examples of measures that can avoid the cancellation properties of singular integrals, providing negative answers to the expected behaviour of some dyadic operators and the Hilbert transform at the endpoint $p=1$. More precisely they prove that the inequality below, also known as Muckenhoupt-Wheeden Conjecture, is false for a dyadic operator and for the Hilbert transform respectively. 

\begin{eqnarray} \label{e.reguera-thiele}
\displaystyle\sup_{t>0} tw (\{x \in \ral : |Tf(x)| > t \}) \lesssim
\nint_{\ral} |f|  M w(x)dx. \end{eqnarray}

\noindent But the endpoint estimate is a delicate one. Let us recall that a weight $w$ is in $A_1$ if there exist constants $C$ such that $Mw(x)\leq C w(x)$ a.e., the infimum of such constants $C$ is denoted by $\norm w.A_{1}.$. The subsequent work of Nazarov, Reznikov, Vasuynin and Volberg \cite{NRVV} shows that there exists a weight $w$ in $A_1$ such that the Hilbert transform, $H$, fails to follow the expected sharp estimate \eqref{e.a1}.  

\begin{equation} \label{e.a1}
\displaystyle\sup_{t>0} tw (\{x \in \ral : |Tf(x)| > t \}) \lesssim 
\norm w. A_{1}. \nint_{\ral} |f|  w(x)dx. \end{equation}
\noindent An explicit construction of the $A_1$ weight $w$ that violates \eqref{e.a1} has not been provided yet.

Further, we wish to emphasize that the weights constructed in \cites{1011.1767, 1008.3943} fail to be $A_p$ weights. Indeed, the aforementioned weights are differentiated from those in the $A_p$ classes by their compact support and absence of a doubling property. In Theorem \ref{m.c.lsut}, we provide an example where the necessary direction in the Muckenhoupt-Wheeden Conjecture \ref{c.lpMW} is not satisfied.  As a result, we conclude there is no a priori relationship between the Hilbert transform $H$ and the maximal operator $M$ in the two weight setting.

The paper is organized as follows: in Section 2 we define some of the concepts and needed background. Section 3 describes the construction of measures $w_k$ from \cite{1011.1767}. Section 4 describes the proof of Theorem \ref{t.maint2}. Section 4 contains the discussion of the examples in the two weight setting that disprove sufficiency, i.e. Theorem \ref{t.main}. Section 6 contains the construction of the measures in \cite{1001.4043} and the proof of Theorem \ref{m.c.lsut}.

\subsection{Acknowledgments} The authors would like to thank Dr. Michael Lacey for pointing out the conjecture, numerous discussions, and his suggestions. The authors would also like to thank Dr. Brett Wick for his time, suggestions, and discussions.

\section{Initial concepts}  

The space we will be working on is $\mathbb R$. A weight $w$ is a non-negative locally integrable function. Throughout the paper $1_{E}$ will be the characteristic function associated to the set $E\subset \mathbb R$. The numbers $p,p'$ will satisfy $1<p,p'<\infty$ and $\frac{1}{p}+\frac{1}{p'}=1$. Throughout the paper we will use the same symbol for a measure $w$ and its Lebesgue density.

In the sequel when referring to $M$ we will understand the Hardy-Littlewood maximal function, i.e., for $f\in L^{1}_{loc}(\mathbb R)$
\begin{equation}\label{e.maxf}
M f(x)=\sup_{Q\,\, \text{cube}}\frac{1_{Q}}{|Q|} \int_{Q} f(y)dy.
\end{equation}
We will also consider the Hilbert transform, 
\begin{equation}\label{e.hilbertt}
Hf(x)= \textup{p.v.} \int_{\mathbb R} \frac{f(y)}{y-x}dy.
\end{equation}

When looking at two weigthed inequalities, we use a different formulation. For $w,v$ two positive Borel measures the statements below, first introduced by E. Sawyer in \cite{MR676801}, are equivalent. \eqref{e.Sawyer} and \eqref{e.dual} represent the self-dual form for two weighted inequalities and we will use them frequently throughout the paper.

\begin{align}
\label{e.general} \norm Tf.L^{p}(w).& \leq  C\norm f.L^{p}(v).,\\
\label{e.Sawyer} \norm T(f\sigma).L^{p}(w).& \leq  C\norm f.L^{p}(\sigma).,\quad \sigma=v^{1-p'}1_{\textup{supp}(v)}.\\
\label{e.dual} \norm T^{*}(fw).L^{p'}(\sigma).& \leq  C\norm f.L^{p'}(w).
\end{align}

A characterization for the boundedness of the Hardy-Littlewood maximal operator in the two weight setting was described by E. Sawyer in \cite{MR676801}. Let us recall the main result in \cite{MR676801}:


\begin{theorem} \label{sawyer}
Let $w$ and $v$ be weights, and let $\sigma$ be the dual weight of $v$, i.e., 
$\sigma=v^{1-p'}$. Then $M$ is bounded from $L^{p}(v)$ to $L^{p}(w)$ if and only if the following inequality holds
\begin{equation}\label{e.testingM}
\int_{Q} \left |M(\sigma 1_{Q})\right |^{p}\leq C \sigma(Q) \quad \text{for all } \,Q \text{ cubes}
\end{equation}
and it is uniformly on cubes. 
\end{theorem}

\section{Construction of weights $w_{k}$}
The construction of measures that exploit lack of cancellation of singular integral operators appears in the work of Reguera \cite{1008.3943}, followed by the simplified construction in Reguera-Thiele's paper \cite{1011.1767}. The measures are supported in residual intervals $I(J)$ associated to a Cantor type construction. The $I(J)$ are selected as intervals on which the Hilbert transform avoids cancellations, and this property is preserved in subsequent steps of the construction by holding the total measure of the neighbour intervals constant and choosing an appropiate location for $I(J)$. 

It is the Reguera-Thiele construction that we present in this paper. For the sake of completeness, we will write the steps of such construction below. We consider triadic intervals, $\mathcal T =\left \{ [3^{j}n, 3^j(n+1))\mid n,j\in \mathbb Z \right\}$. Given an interval $I=[a,b)$, we use the notation $\text{rep}:=b$ and $\text{lep}:=a$ for the boundary points of the interval.
We also denote by ${I^m}$ the triadic child of $I$ which contains the center of $I$. For a fixed integer $k$, we define $\mathbf K_0$ to be $\{[0,1)\}$ and recursively for $i\ge 1$:
$$\mathbf J_i:=\{K^m: K\in \mathbf K_{i-1}\}\ \ ,$$
$$\mathbf K_i:=\{K : K\in \mathcal T,\  |K|=3^{-ik},\  K\subset \bigcup_{J\in\mathbf J_i}J\}\ \ .$$

We define $\mathbf J:=\bigcup_{i\ge 1} \mathbf J_i$ and we use $\{ \epsilon(J) \}_{J \in \mathbf{J}}$ to denote a collection of appropriately chosen signs indexed by $\mathbf J$. In particular, the value of each $\epsilon(J)$ is selected depending on the values $\epsilon(J')$ with $|J'|>|J|$; 
the exact choice is determined in the proof of Lemma \ref{ch3hestimate} below. For a detailed explanation we refer the reader to \cite{1011.1767}.
Notice that for each $J\in\mathbf J$ there is an $i$ such that $J=K^{m}$ with $K\in K_{i}$. Consequently, we may define the residual interval associated to $J$, denoted $I(J)$, to satisfy the following:
\begin{enumerate}
\item $|I| = 3^{-(i+1)k}$
\item $\text{rep}(I(J)) = \text{lep}(J)$ if $\epsilon(J) = 1$ 
\item $\text{lep}(I(J)) = \text{rep}(J)$ if $\epsilon(J) = -1$.
\end{enumerate}
We also mention in passing a useful property of $I(J)$: namely, $I(J)$ has the same length as the intervals $K$ in the next generation $\mathbf K_{i+1}$.

\indent Now we define a sequence of absolutely continuous measures on $[0,1]$.
Let $w_{k}^{0}$ be the uniform measure on $[0,1)^m\cup I([0,1)^m)$
with total mass $1$ and recursively define the measure $w_{k}^{i}$ by
the following properties: it coincides with $w_{k}^{i-1}$ on the complement of 
$\bigcup_{K\in \mathbf K_{i}}K$,
for $K\in \mathbf K_i$ we have $w_{k}^{i}(K)=w_{k}^{i-1}(K)$,
and the restriction of $w_{k}^{i}$ to $K$ is supported and uniformly distributed on 
$K^m\cup I(K^m)$.

Let $w_{k}$ be the weak limit of the sequence $w_{k}^{i}$. Observe 
that $w_{k}$ is supported on $\bigcup_{J\in \mathbf J} I(J)$.

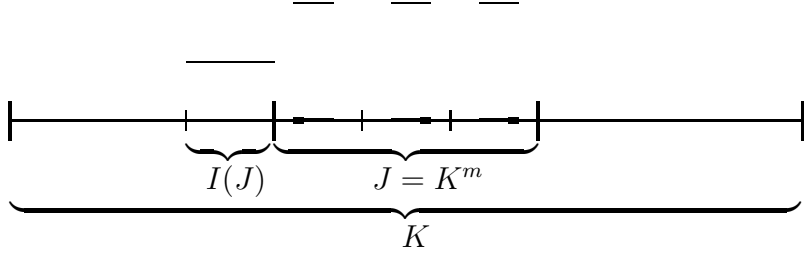
\begin{figure}
\setlength{\unitlength}{1.3mm}
\begin{center}
\begin{picture}(100,23)



\put(37,8){$\underbrace{\rule{3.5cm}{0cm}}$}
\put(10,2){$\underbrace{\rule{10.5cm}{0cm}}$}
\put(28,8){$\underbrace{\rule{1.1cm}{0cm}}$}

\put(50,-3){$K$}
\put(47,3){$J=K^m$}
\put(30,3){$I(J)$}

\put(28,16){\line(1,0){9}}
\put(39,22){\line(1,0){4}}
\put(49,22){\line(1,0){4}}
\put(58,22){\line(1,0){4}}

\put(10,10){\line(1,0){81}}
\multiput(28,9)(9,0){5}{\line(0,1){2}}
\linethickness{1pt}
\multiput(10,8)(27,0){4}{\line(0,1){4}}
\multiput(40,10)(9,0){3}{\line(1,0){3}}
\linethickness{2pt}
\put(39,10){\line(1,0){1}}
\put(52,10){\line(1,0){1}}
\put(61,10){\line(1,0){1}}

\end{picture}
\end{center}
\caption{Second stage of the construction of the measure $w_{k}$ for $k=1$.}
\label{ch3f1.}
\end{figure} 

\begin{lemma}[Reguera-Thiele]\label{ch3hestimate}
For $K\in \mathbf K_{i}$, $J=K^m$ and $k>3000$ we have
\begin{align}
\label{e.hestimate}|Hw_{k}(x)|\ge &(k/3) w_{k}(x) \quad x\in I(J)^m.
\end{align}
\end{lemma}

For the proof of this lemma we refer the reader to \cite{1011.1767}. We can also prove

\begin{equation}
\label{e.contmax} Mw_{k}(x)\le 13 w_{k}(x)\quad x\in I(J).
\end{equation}
There exists a proof of this fact in \cite{1011.1767} for $ x\in I(J)^{m}$ with a smaller constant. For the proof of \eqref{e.contmax} we fix $x\in I(J)$ and consider an interval $I$, not necessarily triadic, such that $x\in I$. If $I\subset I(J)$, then 
$$
\frac{w_{k}(I)}{|I|}= \frac{w_{k}(I(J))}{|I(J)|}=w_{k}(x),
$$ 
and that is enough. Suppose $I\nsubseteq I(J)$ and consider the family $\mathcal L$ of intervals $L$ such that $L\cap I(J)\neq \emptyset$ and $|L|=|I(J)|$. We will distinguish two cases. 
\begin{enumerate}
\item Suppose $|I|\geq \frac{|I(J)|}{6}$, then
\begin{eqnarray*}
\frac{w_{k}(I)}{|I|} &  \leq  &\frac{\sum_{L\in\mathcal L}w_{k}(L)}{|I|}\\
 &\leq & 13 \frac{\sum_{L\in\mathcal L}w_{k}(I(J))}{\sum_{L\in\mathcal L}|I(J)|}\\
 & \leq & 13 \frac{w_{k}(I(J))}{|I(J)|},
\end{eqnarray*}
where we have used the fact that $\sum_{L\in\mathcal L} |L|\leq |I|+2|I(J)|\leq 13|I|$.
\item Suppose now that  $|I|<\frac{|I(J)|}{6}$, then it is easy to see that $w(I\cap I(J)^{c})=0$. Therefore
$$
\frac{w_{k}(I)}{|I|}= \frac{w_{k}(I\cap I(J))}{|I|}=w_{k}(x)\frac{|I\cap I(J)|}{|I|}\leq w_{k}(x).
$$
\end{enumerate}
The proof of \eqref{e.contmax} is concluded. 


\section{One weight case: Proof of \ref{t.maint2}}

First of all, for a weight $w$ that is compactly supported, we will understand $L^p(w)$ is the set of functions that are supported on the support of $w$ and satisfy the corresponding $L^{p}$ bounds.
Next we prove that there exists a weight $u$, for which $H$ does not map $L^p(u)$ into $L^p(u)$ but the Hardy-Littlewood maximal operator does.

\subsection{Unboundedness of the Hilbert transform}

The main result of this subsection is encoded in the following proposition. 

\begin{proposition}\label{p.unbddH1}
For each sufficiently large integer $k>0$, there exists a nontrivial weight $u_{k}$ and a function $f_{k}$, $f_{k}\in L^{p}(u_{k})$ such that
\begin{equation*}
\|H(f_{k})\|_{L^{p}(u_{k})}
\ge \frac{1}{3^{p+1}} k^{p}  \|f_{k}\|_{L^{p}(u_{k})}
\end{equation*}
\end{proposition}

\begin{proof}

Let $w_{k}$ as in Section 3, we define $u_{k}=w_{k}^{1-p}$ and $f_{k}=w_{k}$. Using the estimate \eqref{e.hestimate}, we can conclude 

\begin{align*}
\int \left | H(w_{k}) \right|^{p}w_{k}^{1-p}(x) dx =& \sum_{J\in \mathbf J}\int_{I(J)}\left | H(w_{k}) \right|^{p}w_{k}^{1-p}(x) dx\\
\geq &\sum_{J\in \mathbf J}\int_{I(J)^{m}}\left | H(w_{k}) \right|^{p}w_{k}^{1-p}(x) dx\\
\geq & \sum_{J\in \mathbf J}\left(\frac{k}{3}\right)^{p}\int_{I(J)^{m}} w_{k}(x)dx\\
\geq & \left(\frac{k}{3}\right)^{p}\frac{1}{3} w_{k}([0,1)) = \left(\frac{k}{3}\right)^{p}\frac{1}{3}\norm w_{k}. L^{p}(u_{k}).^{p}
\end{align*}

\end{proof}

The proof of unboundedness of the Hilbert transform as stated in Theorem \ref{t.maint2} will require a gliding hump argument. We line out the main ideas here. Consider the weight $w=\sum_{k=1}^{\infty}w_{k}(x-3^{k})$ and the function $g=\sum_{k=1}^{\infty}\frac{1}{k^{\epsilon}}1_{[3^{k}, 3^{k}+1)}$ with $1/p<\epsilon<1$. Then consider $u=w^{1-p}$ and $f=gw$. It is easy to see that $f\in L^{p}(u)$. We claim
\begin{equation}
\label{e.gliding}
\norm Hf.L^{p}(u).=\infty
\end{equation}

\begin{proof}[Proof of \eqref{e.gliding}]
\begin{eqnarray*}
\norm Hf.L^{p}(u).^{p} &= &\sum_{k}\int_{[3^{k},3^{k}+1)}\left|\frac{1}{k^{\epsilon}}H(w_{k}(\cdot-3^k))+ \sum_{n: n\neq k}\frac{1}{n^{\epsilon}}H(w_{n}(\cdot-3^n))\right|^{p}(x)w_{k}^{1-p}(x-3^k)dx\\
 & \geq & \frac{1}{2}\sum_{k}\int_{[3^{k},3^{k}+1)}\left|\frac{1}{k^{\epsilon}}H(w_{k}(\cdot-3^k))(x)\right|^{p}w_{k}^{1-p}(x-3^k)dx \\
 &=&\frac{1}{2} \sum_{k}\int_{[0,1)}\left|\frac{1}{k^{\epsilon}}H(w_{k})(x)\right|^{p}w_{k}^{1-p}(x)dx\\
& \geq & \frac{1}{3^{p+1}}\sum_{k} \left(\frac{k}{k^{\epsilon}}\right)^{p}=\infty,
\end{eqnarray*}
where we have used that $H$ is translation invariant, the estimate from Proposition \ref{p.unbddH1} and the fact that $\epsilon<1$.
\end{proof}

\subsection{Boundedness of the Maximal function}

In this subsection we will consider boundedness of the Maximal function for the weight $u=w^{1-p}$ where $w=\sum_{k\geq 1}w_{k}(\cdot-3^k)$ considered in the previous subsection. We first need to establish the following proposition.

\begin{proposition} \label{maxbounded1}
Let $1 < p < \infty$. For $M$, the Hardy-Littlewood maximal operator, and $u_{k}$ the weights considered in Proposition \ref{p.unbddH1}, we have
\begin{equation}\label{e.maxbdd1}
M:  L^p \left(u_{k} \right) \mapsto L^p\left(u_{k} \right),
\end{equation}
with norm independent of $k$.
\end{proposition}

\begin{proof}
Let $u_{k}$ be the weight from Proposition \ref{p.unbddH1}. Using Sawyer's characterization, it is enough to check the testing condition \eqref{e.testingM} for any interval $Q$. The estimate for the maximal function \eqref{e.contmax} allows us to prove the desired testing condition.

\begin{align*}
\int_{Q} \left |M(w_{k} 1_{Q})\right |^{p}(x)w_{k}^{1-p}(x)dx&= \sum_{J\in \mathbf J}\int_{I(J)\cap Q}\left |M(w_{k} 1_{Q})\right |^{p} (x)w_{k}^{1-p}(x)dx\\
&\leq 13^{p} \sum_{J\in \mathbf J}\int_{I(J)\cap Q} w_{k}(x)dx\\
& \leq 13^{p} w_{k}(Q).
\end{align*}
which implies the desired conclusion.
\end{proof}

In order to conclude the proof of Theorem \ref{t.maint2}, we need to prove $M(w\cdot): \, L^{p}(w)\mapsto L^{p}(w^{1-p})$ for $w=\sum_{k=1}^{\infty}w_{k}(x-3^{k})$. Once again we use the characterization provided by Sawyer to reduce the problem to checking \eqref{e.testingM} for $Q$ a fixed interval.

\begin{eqnarray*}
\norm 1_{Q}M(w1_{Q}).L^{p}(w^{1-p}).^{p}&=& \sum_{k\geq 1}\int_{Q\cap[3^{k},3^{k}+1)}|M(w1_{Q})|^{p}(x)w_{k}^{1-p}(x-3^{k})dx\\
&\leq &\sum_{k\geq 1}\int_{(Q-3^{k})\cap[0,1)}|M(w)|^{p}(x+3^{k})w_{k}^{1-p}(x)dx\\
&\leq & C\sum_{k\geq 1}\int_{(Q-3^{k})\cap[0,1)}|M(w_{k})|^{p}(x)w_{k}^{1-p}(x)dx\\
&\leq & 13^{p} \sum_{k\geq 1}w((Q-3^{k})\cap[0,1))\leq 13^{p} w(Q).
\end{eqnarray*}

In the previous inequalities, we have used Proposition \ref{maxbounded1} and the fact that $Mw(x+3^{k})\approx Mw_{k}(x)$ when $x\in \text{supp}(w_{k})\subset [0,1)$.
This concludes the proof of Theorem \ref{t.maint2}.


\section{Sufficiency in the two weight case: Proof of Theorem \ref{t.main}}

In this section we consider the weights $w_{k}$ from Section 3 and $v_{k}=\left(\frac{Mw_{k}}{w_{k}}\right)^{p}w_{k}$.We will also need the dual weight of $v_{k}$, $\sigma_k=v_{k}^{1-p'}=\frac{w_{k}}{(Mw_{k})^{p'}}$. The proof of Theorem \ref{t.main} is divided in two subsections, in the first one we prove unboundedness of the Hilbert transform, in the second one, we include the proof of boundedness of the Maximal operator in both directions. In order to show a complete proof of Theorem \ref{t.main} one would need to use gliding hump arguments analogous to the ones used in the previous section, we leave the details to the reader. 

\subsection{Unboundedness of the Hilbert Transform}

We describe the main result of this subsection in the proposition below.

\begin{proposition}\label{p.unbddH2}
Let $1<p<\infty$ and let $p'$ be its dual exponent. For each sufficiently large integer $k>0$, there exist nontrivial weights $w_{k}$ and $\sigma_{k}$ on the real line such that
\begin{equation}\label{e.unbddH2}
\|H(w_{k}1_{[0,1)})\|_{L^{p'}(\sigma_{k})}
\ge k^{p'}\frac{1}{3}\frac{1}{21^{p'}} \|1_{[0,1)}\|_{L^{p'}(w_{k})}
\end{equation}
\end{proposition}

\begin{remark}
To prove unboundedness of the Hilbert transform it is enough to provide a counterexample to the dual inequality \ref{e.dual}. Proposition \ref{p.unbddH2} provides the desired result.
\end{remark}

\begin{remark}
The case $p=2$ was already proven by Reguera and Thiele in \cite{1011.1767}. The proposition extends the result for every $1<p<\infty$.
\end{remark}

\begin{proof}
Let us choose the weight $w_{k}$ as in Section 3 and $\sigma_{k}=\frac{w_{k}}{(Mw_{k})^{p'}}$ . To complete this proof we will use the estimates obtained in Lemma \ref{ch3hestimate}. 

\begin{align*}
\int_{[0,1)}\left | H(w_{k}1_{[0,1)}) \right|^{p'}\sigma_{k} dx =& \sum_{J\in \mathbf J}\int_{I(J)}\left | H(w_{k}1_{[0,1)}) \right|^{p'}\sigma_{k} dx\\
\geq & \sum_{J\in \mathbf J}\int_{I(J)^{m}}\left | H(w_{k}1_{[0,1)}) \right|^{p'}\sigma_{k} dx\\
\geq & \sum_{J\in \mathbf J}\frac{k^{p'}}{21^{p'}}\int_{I(J)^{m}} w_{k}(x)dx\\
\geq & \frac{k^{p'}}{21^{p'}} \frac{1}{3}w_{k}([0,1)).
\end{align*}

\end{proof}


\subsection{Boundedness of the Maximal function}
We devote this subsection to proving the following Proposition:
\begin{proposition} \label{maxbounded}
If $M$ is the Hardy-Littlewood maximal operator, $1 < p < \infty$ and $w$ is a weight, then we have
\begin{eqnarray}
\label{e.pmax} M: &  L^p \left(v \right) \mapsto L^p(w)  \\
\label{e.christsawyer} M:&  L^{p^{\prime}}(w^{1- p^{\prime}}) \mapsto L^{p^{\prime}} ( \sigma ),
\end{eqnarray}
with $v = \left(\frac{M w}{w} \right)^{p} w$ and $\sigma = v^{1- p^{\prime}}$. 
\end{proposition}
Prior to proving Proposition \ref{maxbounded}, we introduce a lemma and some attendant notation. Let $\mathcal D$ be the usual dyadic grid in $\mathbb R$, namely 
$\mathcal D = \{ \left.\left [ 2^{j}m, 2^{j}(m+1)\right.\right ),\,\,\,m, j \in \mathbb Z \}$ and let $\mathcal D_{{\textrm{shift}}}$ denote the shifted dyadic grid of Michael Christ, i.e.
\[ \mathcal D_{{\textrm{shift}}} = \{2^j \left( [n, n+1) + (-1)^j 3^{-1} \right): n, j \in \mathbb{Z} \}. \]
For $f \in L^1_{\textrm{loc}}(\mathbb{R})$, we define 
\begin{eqnarray*}
 \mathcal Mf(x) = \displaystyle\sup_{ I \in \mathcal D} \frac{1_{I}}{| I |} \displaystyle\int_{I} |{f(y)}| dy. 
\end{eqnarray*}
Equivalently we can define $\mathcal M_{\textrm{shift}}f$, where the supremum is taken over intervals in $D_{\textrm{shift}}$.
Then we have, 
\begin{lemma}\label{christ} For any finite interval $I$, there exists an interval $I_d \subset \mathcal D \cup \mathcal D_{\textrm{shift}}$ such that $I \subset I_d$ and 
$| I | \approx | I_d |$. As a consequence, for a function $f \in L^1_{\textrm{loc}}(\mathbb{R})$, the following inequality holds:
\begin{align}
M f(x) \lesssim \mathcal Mf(x) + \mathcal M_{\textrm{shift}}f(x) .\label{maxdom}
\end{align}
\end{lemma}
For a proof of this lemma we refer the reader to \cite{Christcubes}.

With Lemma \ref{christ} in hand, we now proceed to the proof of Proposition \ref{maxbounded}. 
\begin{proof}
\indent The proof of \eqref{e.pmax} follows from an extrapolation argument of D. Cruz-Uribe and C. P\'{e}rez \cite{MR1761362}, so we only need to consider \eqref{e.christsawyer}. Instead of proving $\eqref{e.christsawyer}$ directly, by \eqref{e.Sawyer}, we may verify the following equivalent expression
\begin{eqnarray} M (\cdot w) : L^{p^{\prime}}(w) \rightarrow L^{p^{\prime}}(\sigma) \label{j.dual}, \end{eqnarray}
holds. Consideration of Lemma \ref{christ} implies it is sufficient to demonstrate \eqref{j.dual} for an arbitrary dyadic linearization of the maximal function, i.e. we need to show
\begin{eqnarray}
L(\cdot w):& L^{p^{\prime}}(w) \mapsto L^{p^{\prime}}(\sigma) \label{linear}
\end{eqnarray}
with $L$ a linearization of the maximal function. To this end, let 
\begin{align}
L(f w)(x) = \displaystyle\sum_{I \in \mathcal G} \mathbb{E}_{I}(fw) {1}_{E(I)}(x)
\end{align} where $\mathcal G = \mathcal D \ {\textrm{ or }} \ \mathcal D_{\textrm{shift}}$ and each $E(I)$ satisfies $E(I) \subset I$ and $E(I) \cap E(\tilde{I}) = \emptyset$ if $I \neq \tilde{I}$. Before doing any computations, we invoke Theorem \ref{sawyer} which reduces proving \eqref{linear} to showing 
\begin{eqnarray*}
\norme{{1}_{Q} L( {1}_{Q} w)}_{L^{p^{\prime}}(\sigma)} \lesssim w(Q)^{\frac{1}{p^{\prime}}}
\end{eqnarray*}
for $Q$ a dyadic subinterval of $\mathbb{R}$. Now we fix an interval $Q$ and notice that since $E(I) \cap E(Q) = \emptyset $ for $I \neq Q$,
\begin{eqnarray*}
\norme{L({1}_{Q} w)}^{p^{\prime}}_{L^{p^{\prime}}(\sigma)} &=&
\nint_{Q} L({1}_{Q} w)^{p^{\prime}}(x) \sigma(x) \\
&=& \nint_{Q} \left(\displaystyle\sum_{I \in \mathcal G} \mathbb{E}_{I}( {1}_Q w) {1}_{E(I)}(x) \right)^{p^{\prime}} \sigma(x)   \\
&=& \displaystyle\sum_{I \in \mathcal G} \mathbb{E}_{I}({1}_Q w)^{p^{\prime}}  \sigma(E(I) \cap Q) \\
&=& \displaystyle\sum_{I \in \mathcal G} \left( \frac{w(I \cap Q)}{|I |} \right)^{p^{\prime}} \cdot \sigma(E(I) \cap Q)\\
&=& \displaystyle\sum_{\substack{I \in \mathcal G\\ I\subset Q} } \left( \frac{w(I \cap Q)}{|I |} \right)^{p^{\prime}} \cdot \sigma(E(I) \cap Q) + \displaystyle\sum_{\substack{I \in \mathcal G\\ Q\subset I} } \left( \frac{w(I \cap Q)}{|I |} \right)^{p^{\prime}} \cdot \sigma(E(I) \cap Q).
\end{eqnarray*}
As $\sigma = v^{1- p^{\prime}} = \left( \frac{1}{M w (x)} \right)^{p^{\prime}} w(x)$, 
\begin{eqnarray*}
\sigma(E(I) \cap Q) &\leq& w(E(I) \cap Q) \cdot \min\left \{ \left( \frac{|I |}{w( I )} \right)^{p^{\prime}}, \,\left( \frac{|Q |}{w( Q )} \right)^{p^{\prime}} \right \} .  
\end{eqnarray*}

Consequently,
\begin{eqnarray*}
 \displaystyle\sum_{\substack{I \in \mathcal G\\ I\subset Q} } \left( \frac{w(I \cap Q)}{|I|} \right)^{p^{\prime}} \cdot \sigma(E(I) \cap Q)&\leq&
 \displaystyle\sum_{\substack{I \in \mathcal G\\ I\subset Q} } \left( \frac{w (I )}{| I |} \right)^{p^{\prime}} \left(\frac{ |I | }{w(I )} \right)^{ p^{\prime}} w(E(I) \cap Q) \\
 &=& \displaystyle\sum_{\substack{I \in \mathcal G\\ I\subset Q} } w(E(I) \cap Q) \\
 &\leq& w(Q),
\end{eqnarray*}
and
\begin{eqnarray*}
 \displaystyle\sum_{\substack{I \in \mathcal G\\ Q\subset I} } \left( \frac{w(I \cap Q)}{|I|} \right)^{p^{\prime}} \cdot \sigma(E(I) \cap Q)&\leq&
 \displaystyle\sum_{\substack{I \in \mathcal G\\ Q\subset I} } \left( \frac{w (Q )}{| I |} \right)^{p^{\prime}} \left(\frac{ |Q| }{w(Q )} \right)^{ p^{\prime}} w(E(I) \cap Q) \\
 &\leq& w(Q)|Q|^{p'} \sum_{\substack{I \in \mathcal G\\ Q\subset I} } \frac{1}{|I|^{p'}} \\
 &\leq& 2w(Q),
\end{eqnarray*}
Thus,
\begin{eqnarray*}
\nint_{Q} L({1}_{Q} w)^{p^{\prime}}(x) \sigma(x) &\leq& 3w(Q)
\end{eqnarray*}
which implies the desired result and completes the proof of Proposition \ref{maxbounded}.
\end{proof}

Now, by taking Proposition \ref{p.unbddH2} and Proposition \ref{maxbounded} in concert, and by considering a gliding hump argument like the one discussed in Section 4 we immediately obtain Theorem \ref{t.main}.  

\section{Necessity in the two weight case: Proof of Theorem \ref{m.c.lsut}}

\indent In this subsection, we emphasize the disparity between the Hilbert transform and the maximal function by presenting a pair of measures $\lambda$ and $\gamma$ for which the Hilbert transform acts continuously while the maximal function is unbounded. The measures which we will use are due to Lacey, Sawyer, and Uriarte-Tuero \cite{1001.4043} and we begin by briefly describing their construction. In the interest of clarity we introduce $\gamma$ and some attendant notation by describing the Cantor set's construction. We let $I^0_{1} = [0, 1]$ and for $1 \leq r$ we let $\{ I^r_l \}_{l=1}^{2^r}$ denote the $2^r$ closed intervals (ordered left to right) which remain during the $r^{\textrm{th}}$ stage of the Cantor set's construction; in particular, we have $I^1_1 = [0,\frac{1}{3}]$ and $I^1_2 = [\frac{2}{3}, 1]$, $I^2_1 = [ 0 , \frac{1}{9} ]$, $I^2_2 = [ \frac{2}{9}, \frac{1}{3} ]$, $I^2_3 = [\frac{2}{3}, \frac{7}{9}]$, $I^2_4 = [\frac{8}{9}, 1]$ etc. For each $I^r_l$, the corresponding open middle third interval which is removed during the $r+1$ stage of construction will be denoted by $G^r_l = (a^r_l, b^r_l)$; so, we have $G^0_1 = (\frac{1}{3}, \frac{2}{3})$, $G^1_1 = (\frac{1}{9}, \frac{2}{9} )$, $G^1_2 = (\frac{7}{9}, \frac{8}{9})$ etc. Further, we denote the Cantor set by $E = \cap_{r=1}^{\infty} \cup_{j=1}^{2^r} I^r_j$. The measure $\gamma$ is the Cantor measure, the unique probability measure on $[0,1]$ which satisfies $\gamma(I^r_l) = 2^{-r}$ for all $ r \geq 0$ and $1 \leq l \leq 2^r$. \\
\indent At this point, we would like to describe the measure $\lambda$. However, prior to doing so, we introduce a lemma which lists important properties of $H(\gamma)$ discussed in \cite{1001.4043}: 
\begin{lemma} \label{h.inc}
For any $l,r \in \nat$,
\begin{itemize}
\item[\textit{i.}] $H(\gamma)(x)$ is decreasing monotonically on $G^r_l$. 
\item[\textit{ii.}] $H(\gamma)(x)$ approaches infinity as $x$ approaches $a^r_l$.
\item[\textit{iii.}] $H(\gamma)(x)$ approaches negative infinity as $x$ approaches $b^r_l$. 
\end{itemize}
\end{lemma}
By Lemma \ref{h.inc}, for each $ r \in \nat$ and $1 \leq l \leq 2^r$, there is a point $\zeta^r_l \in G^r_l$ which satisfies $H(\gamma)(\zeta^r_l) = 0$. We define 
\begin{eqnarray*}
\lambda(x) = \displaystyle\sum_{r =0}^{\infty} \displaystyle\sum_{l=1}^{2^r} \delta_{\zeta^r_l}(x) p^r_l 
\end{eqnarray*}
where $p^r_l = \left( \frac{2}{9} \right)^{r}$ for $r \in \nat$ and $1 \leq l \leq 2^r$. With $\lambda$ and $\gamma$ defined, we may now proceed to the proof of Theorem \ref{m.c.lsut}.  

\begin{proof}[Proof of Theorem \ref{m.c.lsut}]
\indent The verification of $\ref{lsut.hilbert}$ is shown in \cite{1001.4043} so we need only consider $\ref{cantor.max}$. We will show for $r \in \nat$ and $l = 1$  that $\displaystyle\int_{I^r_l} M(1_{I^r_l} \gamma)(x)^2 d \lambda$ is unbounded. Fix $ r \in \nat$ and define a collection of sets $\{ \mathcal G_t \}_{t \in \nat}$ in the following way:
$\mathcal G_0 = G^r_1$ and $\mathcal G_t = \displaystyle\bigcup_{s = 1}^{2^{4t}} G^{r+4t}_s$ for $1 \leq t$. Then we have
\begin{eqnarray}
\nint_{I^r_1} M(1_{I^r_1} \gamma)(x)^2 d \lambda(x) &\gtrsim&
\displaystyle\sum_{i=0}^{\infty} \nint_{\mathcal G_i} M (1 _{I^r_1} \gamma)(x)^2 d \lambda(x) \notag \\
&=& \displaystyle\sum_{i=0}^{\infty} \displaystyle\sum_{s=1}^{2^{4i}} \nint_{G^{r+4i}_s} M ( 1_{I^r_1} \gamma)(x)^2 d \lambda(x) \label{c.max.est}. 
\end{eqnarray}
But, by inspection
\begin{eqnarray*}
M ( 1_{I^r_1} \gamma)(\zeta^{r+4t}_s) &\geq& \left( \frac{3}{2} \right)^{r + 4t} 
\end{eqnarray*}
for $t \in \nat $ and $1 \leq s \leq 2^{4t}$. Now, continuing from the above, we obtain
\begin{eqnarray*}
(\ref{c.max.est}) &\geq& \displaystyle\sum_{i=0}^{\infty} \displaystyle\sum_{s=0}^{2^{4i}} \nint_{G^{r+4i}_s} \left( \frac{3}{2} \right)^{2r + 8i} d \lambda(x) \notag \\
&\geq& \displaystyle\sum_{i=0}^{\infty} \displaystyle\sum_{s=1}^{2^{4i}} p^{r+4i}_s \left( \frac{3}{2} \right)^{2r + 8i}  \notag \\
&=& \displaystyle\sum_{i=0}^{\infty} \displaystyle\sum_{s=1}^{2^{4i}} \left( \frac{2}{9} \right)^{r + 4i} \left( \frac{3}{2} \right)^{2r+8i} \notag \\
&=& \displaystyle\sum_{i=1}^{\infty} 2^{-r} \\
&=& \infty .
\end{eqnarray*}
Immediately, we have $\nint_{I^r_1} M (1_{I^r_1} \gamma)(x)^2 d \lambda(x)$ is unbounded, which completes the proof. 
\end{proof}



\end{document}